\documentclass[12pt]{amsart}

\usepackage{amssymb}

\usepackage{enumerate}

\usepackage{graphicx}

\makeatletter
\@namedef{subjclassname@2010}{%
  \textup{2010} Mathematics Subject Classification}
\makeatother

\newtheorem{thm}{Theorem}[section]
\newtheorem{cor}[thm]{Corollary}
\newtheorem{lem}[thm]{Lemma}

\theoremstyle{definition}

\newtheorem{rem}[thm]{Remark}

\numberwithin{equation}{section}

\newcommand{\C}{{\mathbb{C}}}

\newcommand{\N}{{\mathbb{N}}}
\renewcommand{\Re}{{\mathfrak{Re}}}
\renewcommand{\Im}{{\mathfrak{Im}}}

\newcommand{\s}{\sigma}

\begin{document}

\baselineskip=17pt

\title[A zero density result for the Riemann zeta function]{A zero density result for the Riemann zeta function}

\author[H. Kadiri]{Habiba Kadiri}
\address{Department of Mathematics and Computer Science\\
University of Lethbridge\\
4401 University Drive\\
Lethbridge, Alberta\\
T1K 3M4 Canada}
\email{habiba.kadiri@uleth.ca}

\date{}

\begin{abstract}
In this article, we prove an explicit bound for $N(\s,T)$, the number of zeros of the Riemann zeta function satisfying $\Re s\ge \s$ and $0 \le \Im s \le T$.
This result provides a significant improvement to Rosser's bound for $N(T)$ when used for estimating prime counting functions.
\end{abstract}

\subjclass[2010]{Primary 11M06, 11M26; Secondary 11Y35.}

\keywords{Riemann zeta function, zero density, explicit results.}

\maketitle

\section{Introduction}
In recent years, it has become apparent that explicit results concerning prime numbers are required to solve important problems in number theory.
In particular, the impressive works of Ramar\'e \cite{Ram0}, Tao \cite{Tao}, and Helfgott \cite{HH} related to Goldbach's conjecture highlight the need of better explicit bounds for finite sums over primes.
For instance, they make use of \cite{DR}, \cite{Ram2}, \cite{Ram3}, \cite{Ros}, \cite{RS1}, \cite{RS2}, \cite{S}.
Moreover articles of Rosser and Schoenfeld (\cite{Ros}, \cite{RS1}, \cite{RS2}, \cite{RS3}, \cite{S}), Dusart (\cite{Dus0}, \cite{Dus}, \cite{Dus1}, \cite{Dus2}), and Ramar\'e and Rumely \cite{RR} are extensively used in a wide range of fields including Diophantine approximation, cryptography, and computer science.
These results on primes rely heavily on explicit estimates of sums over the non-trivial zeros of the Riemann zeta function.
More precisely, they rely on three key ingredients: a numerical verification of the Riemann Hypothesis (RH), an explicit zero-free region, and explicit bounds for the number of zeros in the critical strip up to a fixed height $T$.

In 1986, van de Lune et al. \cite{VDL} established that RH had been verified for all zeros $\varrho$ verifying $|\Im \varrho|\le H_0$ with $H_0=545\,439\,823$.
In 2011, Platt \cite{Pla0} \cite{Pla} proved that $H_0=30\, 610\, 046\, 000$ is admissible.
Previously, Wedeniwski \cite{Wed} in 2001 and Gourdon \cite{Gou} in 2004 had announced higher values for $H_0$.
As Platt's computations are more rigourous (he employs interval arithmetic), 
we decide to use his value throughout this article: \[H_0 = 3.061\cdot 10^{10}.\]
For the latest explicit results about zero-free regions for the Riemann zeta function, we refer the reader to \cite{Kad} and \cite{For}.

Let $\s\ge 0.55$. 
We consider $N(\s,T)$, the number of zeros of the Riemann zeta function in the region $\s \le \Re s\le 1$ and $0\le \Im s\le T$.
Trivially we have that $N(\s,T)=0$ for all $T\le H_0$. We prove here an explicit bound for $N(\s,T)$ valid in the range $T\ge H_0$.
\begin{thm}\label{main-density}
Let $\s\ge 0.55$ and $T\ge H_0$. 
Let $\s_0$ and $H$ such that $0.5208< \s_0 < 0.9723 , \s_0< \s$, and $10^3 \le H   \le H_0$.
Then there exist $b_1,b_2$, $b_3$, positive constants depending on $\s, \s_0$, $H$, such that:
\[
N(\s,T)  \le b_1 (T-H) + b_2  \log  (TH) + b_3.
\]
The $b_i$'s are defined in \eqref{def-bi}.
\end{thm}
We rewrite this as
$ N(\s,T)  \le c_1  T + c_2   \log T + c_3$, for $T\ge H_0$.
Numerical values of the $b_i$'s and $c_i$'s are recorded at the end of this article in Table \ref{table1}.
For example, for $\s\ge17/20$ and $T\ge H_0$, we have 
\[
N(\s,T)  \le 0.5561 T + 0.7586  \log  T - 268\,658.
\] 
Let $N(T)$ be the number of non-trivial zeros $\varrho $ with imaginary part $0\le \Im \varrho \le T$. 
We recall that Rosser \cite{Ros} proved  
\begin{equation}\label{Rosser}
\left| N(T) - \frac{T}{2\pi}\log\frac{ T}{2\pi e} - \frac78\right| \le  a\log T + b \log\log T + c,
\end{equation}
with $a=0.137, \ b= 0.443,\ c=1.588$.
Note that Rosser's result got recently improved by Trudgian \cite[Corollary 1]{Tru2} with $a=0.111,\ b=0.275,\ c=2.450$. 
A trivial bound for $N(\s,T)$ follows from the inequalities $N(\s,T)\le\frac12N(T)$ and \eqref{Rosser}:
\[
N(\s,T)  \le \frac{T}{4\pi} \log \Big(\frac{ T}{2\pi e} \Big) (1+o(1)).
\]
Note that when $T$ is asymptotically large, then a factor of $\log T$ is saved. 
Moreover, we have $c_1 \sim \frac{\log \zeta(2\s_0)}{4\pi(\s-\s_0)}$ where $\s_0$ is a parameter which value can be chosen to make $c_1$ as small as possible.
Another feature of Theorem \ref{main-density} is the factor $T-H$: when $T$ is near $H_0$, we choose $H$ to be close to $H_0$ so as to make $N(\s,T)$ of size $\log H_0 $.
This saves a factor of size $H_0$.
As an example, for $\s\ge 17/20$ and $T=H_0+1$, we choose $H=H_0-1$ and $\s_0$ as in Table \ref{table1} and obtain $N(\s,H_0+1)\le 156 $ while \eqref{Rosser} gives $5.2\cdot 10^{10}$ (with either Rosser's or Trudgian's values).

The key motivation for establishing Theorem \ref{main-density} is to use it in place of \eqref{Rosser} and thus to provide improved explicit bounds for Chebyshev's prime counting functions.
We prove in \cite{FK} that, for all $x\ge e^b$,
\begin{equation}\label{psi}
|\psi(x)-x| \le \epsilon_b x,
\end{equation}
where $b$ is a fixed positive constant, and $\epsilon_b$ is an effective positive constant.
For example, for $x\ge e^{50}$ we obtain $\epsilon_{50} =9.461\cdot10^{-10}$ while Dusart \cite[Theorem 2]{Dus1} obtained $0.905\cdot10^{-7}$.

Despite a very rich history of asymptotic results, there were almost no explicit bounds for $N(\s,T)$.
Ramar\'e proved in an unpublished manuscript \cite{Ram} that, for $T\ge 2000$, $Q\ge 10$, and $T\ge Q$,
\[
\sum_{q\le Q} \sum_{\chi \bmod^* q} N(\s,T, \chi) \le 157 (Q^5T^3)^{1-\s} \log^{4-\s}(Q^2T) + 6 Q^2\log^2(Q^2T),
\]
where $\sum_{\chi \bmod^* q}$ denotes the sum over primitive Dirichlet characters $\chi$ to the modulus $q$, and $N(\s,T, \chi) $ counts the number of zeros $\varrho$ of the Dirichlet $L$-function $L(s,\chi)$ 
satisfying $ \s < \Re \varrho<1$ and $ 0 < \Im \varrho < T$.
Taking $Q=10$ and restricting the left sum to $q=1$, it follows that
\begin{equation}\label{Ram1}
N(\s,T ) \le 157 (100\,000T^3)^{1-\s} \log^{4-\s}(100T) + 600\log^2(100T).
\end{equation}
Our main theorem improves Ramar\'e's result for certain values of $\s$ and $T$:
he obtains $N(17/20,10 \cdot H_0) \le 2.675\cdot10^{12}$ while we have $N(17/20,10\cdot H_0) \le 3.404\cdot 10^{10}$.
In 2010, Cheng \cite{Che2} obtained the weaker result:
\begin{equation}\label{Ch1}
N(\s,T) \le 453\,472.54\, T^{8/3(1-\s)} (\log T)^5,
\end{equation}
for all $\s\ge 5/8$ and $T\ge \exp(\exp(18))\simeq 10^{28\,515\,762}$.
His method is based on Ford's \cite{For} effective version of Korobov-Vinogradov's bound for the Riemann zeta function.
He applied \eqref{Ch1} to deduce explicit results on primes between consecutive cubes.
Note that Cheng's result is not valid in the region $T \le \exp(\exp(18))$ while most applications require bounds for $T$ as small as $H_0$.

In order to prove Theorem \ref{main-density} we establish two intermediate theorems about $\zeta(s)$ in the critical strip:
an effective version of a Dirichlet polynomial approximation, and an explicit estimate for the second moment.
\begin{thm}\label{exp-app-fun-equ}
Let $t_0>0$, $s = \sigma + it$ with $\sigma \ge 1/2$, $t\ge t_0$ and $c>\frac1{2\pi}$. Then
\[
\zeta(s) = \sum_{1\le n <  ct} \frac1{n^s} + R(s)
\]
with $ |R(s)| \le  C(\s,c) t^{-\s}$, and
\begin{equation}\label{def-C}
C(\s,c) =
 \Big( c + \frac{1}2 +  \frac{3\sqrt{1+1/t_0^2}}{2\pi} \Big( \frac{ \zeta(2)}{2\pi c} + 1+ \frac1{2\pi c -1}\Big) \Big) c^{-\sigma} .
\end{equation}
\end{thm}
We apply the theorem for $c=1$ and for $t_0$ the height of the first zero of zeta.
\begin{cor} 
Let $\sigma \ge 1/2$ and $t\ge 14.1347$. Then
\begin{equation}\label{approx-ineq-zeta}
\Big|\zeta(s) - \sum_{1\le n < t} \frac1{n^s} \Big| \le c_0 t^{-\s} ,
\quad \text{ where }
\ c_0= 2.1946.
\end{equation}
\end{cor}
This is to compare to Proposition 1 of Cheng \cite{Che1} who obtained $5.505$ instead of $2.1946 t^{-\s}$.
When $\sigma \ge 1/2$ and $0\le t\le 15$, a Mathematica computation gives us that $|\zeta(s) - \sum_{1\le n < t} n^{-s} | \le 43 t^{-\s}$.
\begin{thm}\label{moment-zeta}
Let $ 0.5208< \s_0< 0.9723$ and $10^3 \le H \le H_0$.
We define
\begin{align}
\label{def-epsilon1}
\epsilon_1(\s_0,H) = &  \frac{4H_0}{H_0-H} \Big(   \frac{(\log H_0)H_0^{1-2\s_0}}{2(1-\s_0)}  -  \frac{(2\s_0-1)\log H_0}{2(1-\s_0)H_0} 
\Big. \\ \Big. &
 \nonumber 
+ \frac{\max \Big(0, \frac{1-3\s_0+3\s_0^2}{2(1-\s_0)^2} - \frac{\zeta(2\s_0) }2 \Big)}{H_0} 
+  \frac{(2-\s_0) H_0^{1-2\s_0}  }{2(1-\s_0)^2}  
\Big. \\ \Big. & 
 \nonumber
- \frac{\s_0 H_0^{-\s_0}}{(1-\s_0)^2}  + \frac{H_0^{-2\s_0}}{2(2\s_0-1)}  + \frac{H_0^{-2\s_0-1} }2  \Big) ,
\\ \label{def-epsilon2}
 \epsilon_2(\s_0,H) =& \frac{c_0 ^2}{2\s_0-1} \frac{H^{-(2\s_0-1)}-H_0^{-(2\s_0-1)}}{H_0-H} ,
\\  \label{def-epsilon3}
 \epsilon_3(\s_0,H) =&   2 \sqrt{ \epsilon_2(\s_0,H) (\zeta(2\s_0) +  \epsilon_1(\s_0,H) )},
\\  \label{def-E1}
 \mathcal{E}_1  =&  \epsilon_1 + \epsilon_2  + \epsilon_3 .
\end{align}
Then, for all $T\ge H_0$, we have
 \begin{align*}
& \frac1{T-H} \int_H^T \left| \zeta(\s_0+it)\right|^2 dt \le  \zeta(2\s_0)  + \mathcal{E}_1(\s_0,H) ,
\\ \text{and }\ 
& \int_{H}^{T} \log |\zeta(\s_0+it)| dt \le  \frac{T-H}{2} \log \Big(  \zeta(2\s_0)  + \mathcal{E}_1(\s_0,H) \Big).
 \end{align*}
\end{thm}
For the rest of this article $H$, $T$, $\s_0$, and $\s$ satisfy
\begin{equation}
\label{conditions}
H_0 = 3.061\cdot 10^{10},
10^3 \le H   \le H_0 \le T,
0.5208< \s_0 < 0.9723 ,
\s_1 = 1.5002,
\s_0< \s < \s_1 .  
\end{equation}
%%%%
%%%%
\section{Approximate formula for $\zeta(\sigma+it)$ - Proof of Theorem \ref{exp-app-fun-equ}}
Let  $s=\s+it$ with $1/2<\s<1$ and $t\ge 2$. Let $x=ct$ with $c>\frac{1}{2\pi}$, and let $N$ be a positive integer.
Theorem \ref{exp-app-fun-equ} gives an explicit version of an approximation formula for zeta, as proven by Hardy and Littlewood in \cite{HL}. 
\begin{proof}
We start with the classical identity \cite[equation 3.5.3]{Tit}
 \begin{equation}\label{eq-form1-zeta}
\zeta(s) - \sum_{1\le n < x} \frac1{n^s} =\sum_{x\le n \le N} \frac1{n^s} + s \int_N^{\infty} \frac{((u))}{u^{s+1}}du - \frac{N^{1-s}}{1-s} -\frac12N^{-s},
\end{equation}
where $((u))=[u]-u+1/2$.
The summation formula \cite[equation 2.1.2]{Tit} gives
\[
\sum_{x \le n < N} \frac1{n^s}  
=  \int_x^{N} \frac{du}{u^s} 
  - \frac{((x))}{x^s} + s \int_x^{N} \frac{((u))}{u^{1+s}} du 
 =  \frac{N^{1-s}-x^{1-s}}{1-s}  - \frac{((x))}{x^s} + s \int_x^{N} \frac{((u))}{u^{1+s}} du .
\]
We have the bounds
\[
 \left|\frac{x^{1-s}}{1-s} \right| \le \frac{x^{1-\s}}{t}, \ 
 \left| \frac{((x))}{x^s} \right|\le  \frac{x^{-\s}}{2},\ 
 \left|s \int_N^{\infty} \frac{((u))}{u^{s+1}}du\right| \le \frac{|s|}2 \int_N^{\infty} \frac{1}{u^{\s+1}}du = \frac{|s|}{2\s N^{\s}}.
\]
Thus
\begin{equation}\label{eq-form2-zeta}
\Big|\zeta(s) - \sum_{1\le n < x} \frac1{n^s} \Big|
\le   x^{1-\s} t^{-1}
+ \frac{x^{-\s}}{2}
 + \Big| s \int_x^{N} \frac{((u))}{u^{1+s}} du  \Big|
 +  \frac{|s|}{2\s}  N^{-\s}
+\frac12N^{-\s}.
\end{equation}
The choice $x=ct$ is made to balance the error term $x^{1-\s} t^{-1}+ \frac{x^{-\s}}{2}$.
We appeal to the Fourier series of $((x))$ to obtain a smaller bound for the integral expression. 
For $u\notin\N$, we have \cite[p. 74]{Tit}
\[
((u))
=[u]-u+1/2 
= \frac1{\pi} \sum_{\nu=1}^{\infty} \frac{\sin(2\pi \nu u)}{\nu}.
\]
Lebesgue's bounded convergence theorem applies, and we can exchange the order of the integral and the summation.
We obtain
\begin{equation}\label{ineq1}
 \int_x^{N} \frac{((u))}{u^{1+s}} du 
 = \frac1{\pi} \sum_{\nu=1}^{\infty} \frac{1}{\nu} \int_x^{N} \frac{\sin(2\pi \nu u)}{u^{1+s}} du \\
=  \sum_{\nu=1}^{\infty} \frac{I( \nu) -I(- \nu) }{\nu} ,
\end{equation}
where the integral $I$ is given by 
\begin{equation}\label{ineq2}
I(h) 
= \frac1{2\pi i} \int_x^{N} \frac{ e^{  2i\pi (hu- \frac{t \log u}{2\pi}) }}{u^{\s+1}} du
 = \frac1{ 2\pi } \int_x^{N} F(h,u) \, d (e^{2\pi i( f(u)+h u)} )
\end{equation}
with
$F(h,u) =  \frac{u^{-\s}}{t - 2\pi u h}$ 
and $f(u) = -\frac{t \log u}{2\pi}$.
Since
\[
\frac{\partial}{\partial u} F(h,u)
= u^{-\sigma}\frac{  -\s t u^{-1}+2\pi h (\sigma +1)}{(t - 2\pi u h)^2} ,
\]
it is easy to check that $F(-\nu,u)$ is positive and decreases with $u$, and that $F(\nu,u)$ is negative and increases with $u$.
\\
We now apply the second mean value theorem from \cite[section 12.3]{Titb}:
\begin{lem} 
If $j(x)$ is integrable over $(a,b)$, and $\phi(x)$ is positive, bounded, and non-increasing, then there exists $\xi\in (a,b)$ such that
\[\int_a^b \phi(x) j(x)dx = \phi(a+0)\int_a^{\xi}j(x)dx.\]
\end{lem}
First, we consider $I(-\nu) $.
We separate the real and imaginary part in $d\Big(e^{2\pi i( f(u)+h u)} \Big)$ in \eqref{ineq2} and
we apply the Lemma for $\phi(u) = F(-\nu,u) $. We consider $j(u)du=d(\cos(2\pi ( f(u)-\nu u)))$, and $j(u)du= d(\sin(2\pi ( f(u)-\nu u)))$ respectively.
We obtain that there exist $\xi_1,\xi_2 \in (x, N)$ such that
\begin{multline*}
2\pi I(-\nu) 
= F(-\nu,x) \cos(2\pi ( f(\xi_1)-\nu \xi_1)) - F(-\nu,x) e^{2\pi i ( f(x)-\nu x) } 
\\+ i F(-\nu,x) \sin(2\pi ( f(\xi_2)-\nu \xi_2)).
\end{multline*}
It follows that
\begin{equation}
\label{ineq3}
\left| I(-\nu) \right| 
 \le \frac{3}{2\pi}F(-\nu,x)
=   \frac{3}{2\pi} \frac{(ct)^{-\sigma} }{t + 2\pi ct \nu}
 \le   \frac{3}{(2\pi)^2}\frac{ c^{-\sigma-1} t^{-\sigma-1} }{\nu}.
\end{equation}
A similar argument applies to $I(\nu) $. 
We obtain
\begin{equation}\label{ineq4}
\left| I(\nu) \right| \le  - \frac{3}{2\pi}F(\nu,ct)
=   \frac{3}{2\pi} \frac{(ct)^{-\sigma} }{2\pi ct \nu-t}
\le 
\begin{cases}
\frac{3}{2\pi} \frac{ c^{-\sigma} t^{-\sigma-1} }{\nu-1} \text{ if }\nu\ge 2,\\
\frac{3}{2\pi}\frac{  c^{-\sigma} t^{-\sigma-1} }{2\pi c -1} \text{ if }  \nu=1.
\end{cases}
\end{equation}
Using the simplification 
$ \sum_{\nu=2}^{\infty}\frac{1}{\nu(\nu-1)} = 1$,
$\sum_{\nu=1}^{\infty}\frac{1}{\nu^2} = \zeta(2)$, and  
$\frac{|s|}{t}\le \sqrt{1+1/t^2}$, 
we put together \eqref{ineq1}, 
\eqref{ineq3}, and \eqref{ineq4}, and obtain the bound
\begin{multline*}
\Big| s \int_x^{N} \frac{((u))}{u^{1+s}} du \Big|
\le |s|  \sum_{\nu=1}^{\infty} \frac{ |I( \nu)| + |I(- \nu)| }{\nu} 
\\ \le \frac{3\sqrt{1+1/t^2}}{2\pi}   \Big( 1+ \frac1{2\pi c -1}+ \frac{ \zeta(2)}{2\pi c} \Big) c^{-\sigma} t^{-\sigma} .
\end{multline*}
Letting $N\to\infty$, inequality \eqref{eq-form2-zeta} becomes
\[\Big|\zeta(s) - \sum_{1\le n < ct} \frac1{n^s} \Big|
 \le 
 \Big( c
+ \frac{1}2
 +  \frac{3\sqrt{1+1/t^2}}{2\pi} \Big(1+ \frac1{2\pi c -1}+ \frac{ \zeta(2)}{2\pi c} \Big)
 \Big) (ct)^{-\sigma}  .
\]
 \end{proof}
\begin{rem}
A careful reading of Cheng's proof shows that his error term has size $\mathcal{O}(t^{1-2\s})$, instead of our $\mathcal{O}(t^{-\s})$.
This comes from he fact that he bounds directly the terms $\frac{N^{1-s}}{1-s}$, instead of eliminating them as we did.
\end{rem}
%%%%%%%
\section{Explicit upper bound for the second moment of zeta - Proof of Theorem \ref{moment-zeta}}
We recall that $\s_0, T, H$ are as in \eqref{conditions}.
By Theorem \ref{exp-app-fun-equ}, 
we have the identity
\begin{equation}\label{2nd-moment-zeta}
 \frac1{T-H} \int_H^T \left|\zeta(\s_0+it)\right|^2 dt 
 = D(\s_0,T,H) + E_1(\s_0,T,H)+ E_2(\s_0,T,H)+ E_3(\s_0,T,H),
\end{equation}
where
\begin{align*}
 & 
D(\s_0,T,H) = \frac1{T-H} \int_H^T   \sum_{1\le n < t} \frac1{n^{2\s_0}} dt , \\
& 
E_1(\s_0,T,H) =  \frac2{T-H} \int_H^T \sum_{1\le n <m< t} \frac{\cos(t\log(m/n))}{(nm)^{\s_0} } dt , \\
& 
E_2(\s_0,T,H) = \frac1{T-H} \int_H^T |R(\s_0+it)|^2 dt,\\
& 
E_3(\s_0,T,H) = \frac2{T-H}  \Re  \int_H^T   \sum_{1\le n < t}\frac{R(\s_0+it)}{n^{\s_0+it}} dt .
\end{align*}
We recall here some basic inequalities that we use throughout the following argument.
Let $A,B\in \N$. If $f$ is decreasing and positive, then 
\begin{equation}\label{ineq-sum-int} 
 \sum_{A\le j\le B} f(j) \le f(A) + \int_A^{B} f(u) du.
\end{equation}
For $\s_0>1/2$, 
we bound trivially the diagonal term:
\begin{equation}\label{bound-D}
D(\s_0,T,H) 
 \le  \zeta(2\s_0).
\end{equation}
%%%
We interchange summation order in the off-diagonal terms $E_1(\s_0,T,H)$ and use the fact that $\int_u^v \cos(at)dt \le \frac{2}{a}$ when $a\not=0$:
\[
E_1(\s_0,T,H) 
\le \frac4{T-H} \sum_{1\le n <m< T} \frac{(nm)^{-\s_0}}{ \log(m/n)} .
\]
We use the fact that, for $\lambda>1$ and $\s <1$, 
$\frac1{\log\lambda} 
\le 1+\frac{\lambda^{1-\s}}{\lambda-1}$. Taking $\lambda=\frac{m}{n}$, we obtain
\begin{equation}\label{bound1-E1}
E_1(\s_0,T,H) 
\le \frac4{T-H}  \sum_{1\le n <m< T}  (nm)^{-\s_0}  
+ \frac4{T-H}  \sum_{1\le n <m< T} \frac{m^{1-2\s_0}}{ m-n} .
\end{equation}
For the first sum, we complete the square
\[
\sum_{1\le n <m< T} (nm)^{-\s_0} 
 = \frac12 \Big(\sum_{k<T} k^{-\s_0} \Big)^2 - \frac12 \sum_{k<T}  k^{-2\s_0}
 = \frac12 \Big(\sum_{k<T} k^{-\s_0} \Big)^2 - \frac12\Big( \zeta(2\s_0) - \sum_{k \ge T}  k^{-2\s_0} \Big),
\]
and use \eqref{ineq-sum-int} with $f(t)=t^{-\s_0}$ and $f(t)=t^{-2\s_0}$ to bound the resulting sums. We obtain
\begin{equation}\label{ineq-1stpart-E1}
\sum_{1\le n <m< T} (nm)^{-\s_0} 
\le 
\frac{T^{2(1-\s_0)}}{2(1-\s_0)^2} 
- \frac{\s_0T^{1-\s_0}}{(1-\s_0)^2}
 + \frac{\s_0^2}{2(1-\s_0)^2} - \frac12 \zeta(2\s_0) 
- \frac{T^{1-2\s_0}}{2(1-2\s_0)} 
+ \frac{T^{-2\s_0}}2  .
\end{equation}
We consider $k=m-n$ and separate variables in the second sum of \eqref{bound1-E1} and use \eqref{ineq-sum-int}, with $f(t)=t^{1-2\s_0}$ and $f(t)=t^{-1}$, to bound the resulting sums:
\begin{multline}\label{ineq-2ndpart-E1}
\sum_{1\le n <m< T} \frac{m^{1-2\s_0}}{ m-n}
 \le \Big(\sum_{1\le m< T} m^{1-2\s_0} \Big) \Big(\sum_{1\le  k <T}  k^{-1} \Big) 
\\ \le \frac{(\log T)T^{2(1-\s_0)}}{2(1-\s_0)} + \frac{T^{2(1-\s_0)}}{2(1-\s_0)} +\log T + 1 - \frac1{2(1-\s_0)} -\frac{\log T}{2(1-\s_0)}.
\end{multline}
Together with \eqref{bound1-E1}, \eqref{ineq-1stpart-E1} and \eqref{ineq-2ndpart-E1}, we obtain
\begin{multline*}\label{ineq2-E1}
E_1(\s_0,T,H) 
\le \frac{4T}{T-H} \Big( 
 \frac{(\log T)T^{1-2\s_0}}{2(1-\s_0)} 
-  \frac{2\s_0-1}{2(1-\s_0)} \frac{\log T}{T}
+ \Big( \frac{1-3\s_0+3\s_0^2}{2(1-\s_0)^2} - \frac12 \zeta(2\s_0) \Big)\frac1T 
\Big. \\ \Big. +  \frac{2-\s_0}{2(1-\s_0)^2}  T^{1-2\s_0} 
- \frac{\s_0T^{-\s_0}}{(1-\s_0)^2} 
+ \frac{T^{-2\s_0}}{2(2\s_0-1)} 
+ \frac12 T^{-2\s_0-1} 
\Big).
\end{multline*}
We denote
\begin{align*}
& E_{11}(\s_0,T) =\frac{(\log T)T^{1-2\s_0}}{2(1-\s_0)} 
-  \frac{2\s_0-1}{2(1-\s_0)} \frac{\log T}{T} ,\\
& E_{12}(\s_0,T) = \Big( \frac{1-3\s_0+3\s_0^2}{2(1-\s_0)^2} - \frac12 \zeta(2\s_0) \Big)\frac1T ,\\ 
& E_{13}(\s_0,T) = \frac{2-\s_0}{2(1-\s_0)^2}  T^{1-2\s_0} 
- \frac{\s_0T^{-\s_0}}{(1-\s_0)^2} ,\\ 
& E_{14}(\s_0,T) = \frac{T^{-2\s_0}}{2(2\s_0-1)} 
+ \frac12 T^{-2\s_0-1} ,
\end{align*}
and we now study their behavior with respect to $T\ge H_0$. 
It is immediate that $E_{14}$ decreases with $T$. 
Considering the fact that 
$\frac{1-3\s_0+3\s_0^2}{2(1-\s_0)^2} - \frac12 \zeta(2\s_0)$
changes sign at $\s_0 = 0.679785\ldots$, we obtain
\[
E_{12}(\s_0,T) \le \max \big(0,E_{12}(\s_0,H_0) \big).
\]
For $0.5208 < \s_0 <1$, we find
\[
\frac{\partial E_{11}(\s_0,T)}{\partial T} 
 =
 \frac{- \big(T^{2(1-\s_0)}-1\big) \big( (2\s_0-1)(\log T) - 1 \big) + 2(1-\s_0) }{2(1-\s_0)T^2} 
\le 0,
\]
and, when $\s_0\le 0.9723$, that
\[
\frac{\partial  E_{13}(\s_0,T)}{\partial T} 
= \Big( - \frac{ (2-\s_0) (2\s_0-1) }{2}  T^{1-\s_0} + \s_0^2 \Big) \frac{T^{-1-\s_0}}{(1-\s_0)^2}
\le 0.
\] 
Thus $E_{11}(\s_0,T)$ and $E_{13}(\s_0,T)$ decrease with $T\ge H_0$.
We conclude that, for $T\ge H_0$ and $0.5208\le \s_0\le 0.9723$,
\begin{multline}
\label{bound-E1}
E_1(\s_0,T,H) 
\le \frac{4H_0}{H_0-H} \Big( 
 \frac{(\log H_0)H_0^{1-2\s_0}}{2(1-\s_0)} 
-  \frac{2\s_0-1}{2(1-\s_0)} \frac{\log H_0}{H_0}
\Big. \\ \Big. 
+ \frac{\max \Big(0, \frac{1-3\s_0+3\s_0^2}{2(1-\s_0)^2} - \frac{\zeta(2\s_0) }2 \Big)}{H_0} 
+  \frac{(2-\s_0) H_0^{1-2\s_0}  }{2(1-\s_0)^2}  
- \frac{\s_0 H_0^{-\s_0}}{(1-\s_0)^2} 
+ \frac{H_0^{-2\s_0}}{2(2\s_0-1)} 
+ \frac{H_0^{-2\s_0-1} }2 
\Big).
\end{multline}
Theorem \ref{exp-app-fun-equ} gives
\begin{equation}\label{bound-E2}
E_2(\s_0,T,H) 
 \le c_0^2 \frac1{T-H}  \int_H^T t^{-2\s_0} dt  
 \le   \frac{c_0 ^2}{2\s_0-1} \frac{H^{-(2\s_0-1)}-H_0^{-(2\s_0-1)}}{H_0-H}.
\end{equation}
We use the Cauchy-Schwarz inequality to bound $E_3$:
\begin{align}
E_3(\s_0,T,H) 
&\nonumber  \le 2\Big( \frac1{T-H}  \int_H^T |\Re R(s)|^2 dt \Big)^{\frac12} \Big( \frac1{T-H}  \int_H^T \Big| \sum_{1\le n < t} \frac1{n^{\s_0+it}} \Big|^2dt \Big)^{\frac12}
\\\nonumber & \le 2\sqrt{E_2(\s_0,T,H)\big(   D(\s_0,T,H) + E_1(\s_0,T,H) \big)}
\\ & \le 2 \sqrt{ \epsilon_2(\s_0,H) \big(\zeta(2\s_0) +  \epsilon_1(\s_0,H) \big)} .
\label{bound-E3}
\end{align}
The definitions of $\epsilon_1,\epsilon_2,\epsilon_3$ follow from \eqref{bound-E1}, \eqref{bound-E2}, and \eqref{bound-E3}.
The proof is achieved by putting together \eqref{2nd-moment-zeta}, \eqref{bound-D}, \eqref{bound-E1}, \eqref{bound-E2}, \eqref{bound-E3}, and by applying the following bound for concave functions
 \[
  \int_{H}^{T} \log |\zeta(\s_0+it)| dt
 \le \frac{T-H}2 \log \Big( \frac1{T-H} \int_H^T \left| \zeta(\s_0+it)\right|^2 dt \Big).
\]
\section{A lower bound for $ \log |\zeta(s)| $ when $\s>1$.}
\begin{lem}\label{bound-logzeta}
Let $2\le H\le T$ and $\s_1 = 1.5002$. Then
\begin{equation}\label{def-E2}
 \int_H^T \log|\zeta(\s_1+it)| dt  \ge - \mathcal{E}_2 , 
\text{ with }
\mathcal{E}_2 =  1.7655.
\end{equation}
\end{lem}
\begin{proof}
Let $s=\s_1+it$.
It follows from the Euler product that   
\[
 \log |\zeta(s)| 
=  \Re \sum_{n\ge2} \frac{\Lambda(n)}{(\log n)n^{s}} .
\]
Thus
\[  \int_H^T \log|\zeta(\s_1+it)| dt 
= \sum_{n\ge 2} \frac{\Lambda(n) \big(\sin(T\log n)-\sin(H\log n)\big) }{ (\log n)^2 n^{\s_1} }  
\ge - 2\sum_{n\ge 2} \frac{\Lambda(n)}{ (\log n)^2 n^{\s_1}} .
\]
We truncate the sum at $N_0=10^3$ and bound the tail
\[
\sum_{n> N_0}\frac{\Lambda(n)}{ (\log n)^2 n^{\s_1}} 
\le 
\frac1{(\log N_0)^2}\Big( -\frac{\zeta'}{\zeta}(\s_1)- \sum_{n\le N_0}\frac{\Lambda(n)}{n^{\s_1}} \Big) .
\]
We obtain
\[\int_H^T \log|\zeta(\s_1+it)| dt 
\ge -2\Big(  \frac{-\frac{\zeta'}{\zeta}(\s_1) }{(\log N_0)^2}
+ \sum_{n\le N_0}\frac{\Lambda(n)}{n^{\s_1}}\Big(\frac1{(\log n)^2}-\frac1{(\log N_0)^2}\Big) \Big) ,
\]
and a numerical calculation with Maple gives the value for the above left term.
\end{proof}
\section{Explicit bounds for $ \int_{\s_0}^{\s_1} \arg\zeta(\tau+iT)d\tau$.}
\begin{lem}\label{bound-arg}
Let $\eta=0.0001, \s_1 = 3/2+2\eta=1.5002$. Let $\s_0, T, H$ satisfy
$ \s_0 < \s_1 , 2 \le H  \le T$. 
Then
\[
 \int_{\s_0}^{\s_1} \arg\zeta(\tau+iT)d\tau -  \int_{\s_0}^{\s_1} \arg\zeta(\tau+iH)d\tau
\le \mathcal{E}_3(\s_0) \log (H T) + \mathcal{E}_4(\s_0,H) 
\]
with
\begin{align}
& \label{def-E3}
 \mathcal{E}_3(\s_0) =   \frac{\pi (1+2\eta) (\s_1-\s_0)}{4\log2} ,\\ 
& \label{def-E4}
\mathcal{E}_4(\s_0,H) =   \frac{\pi(\s_1-\s_0) }{\log2}  \log\Big(   3\frac{H+3(1+\eta) }{H- (1+2\eta)} \Big(\frac{3(1+\eta)/H +1}{2\pi}\Big)^{\frac{1+2\eta}2} 
\frac{\zeta(1+\eta)^4} {\zeta(2(1+\eta))^2}  \Big).
 \end{align}
\end{lem} 
It suffices to bound an integral of the form
\[ \int_{\s_0}^{\s_1} \arg\zeta(\tau+it)d\tau, \]
with $t\ge H$.
We only make use of the convexity bound for $\zeta(s)$.
\begin{proof}
Let $\omega\in\C$ and $N\in \N$.
Following Rosser's modification of Backlund's trick (\cite[equation (32)]{Back} and \cite[page 223]{Ros}), we introduce
$\displaystyle{ f_t(\omega)= \frac12\Big( \zeta(\omega+i t)^N+\zeta(\omega- i t)^N \Big) }$. 
We denote $n$ to be the number of real zeros of $f_t(\tau) = \Re \zeta(\tau+i t)^N$ in the interval $\s_0< \tau <\s_1$. 
The interval is split into $n+1$ subintervals and on each of them $ \arg \zeta(\tau+it)^N$ changes by at most $\pi$.
Thus
\begin{equation}\label{ineq-arg}
\left|  \int_{\s_0}^{\s_1} \arg\zeta(\tau+it)d\tau  \right|
= \frac1N\left| \int_{\s_0}^{\s_1}  \arg \zeta(\tau+it)^N d\tau\right|
 \le \frac{ (\s_1-\s_0)(n+1)\pi}{N} .
 \end{equation}
We denote $n(r)$ the number of zeros of $ f_t$ in the circle centered at $1+\eta+it$, and with radius $r$. 
For $r\ge 1/2+\eta$, the segment $[\s_0,\s_1]$ is contained in $[1+\eta-r, 1+\eta+r]$, thus $n \le n(r)$.
The following version of Jensen's formula \cite[p. 137, equation (2)]{Ste},
\[
\log |f_t(1+\eta)| + \int_0^{1+2\eta} \frac{n(r)}{r} dr = \frac1{2\pi} \int_{-\pi/2}^{3\pi/2} \log\left| f_t(1+\eta+(1+2\eta)e^{i\theta}) \right|d\theta,
\]
allows us to deduce an upper bound for $n$:
\begin{equation}\label{bound0}
n \le \frac1{\log2} \int_0^{1+2\eta} \frac{n(r)}{r} dr \le \frac1{2\pi\log2} \int_{-\pi/2}^{3\pi/2} \log\left| f_t(1+\eta+(1+2\eta)e^{i\theta}) \right|d\theta - \frac{\log |f_t(1+\eta)|}{\log2}. 
\end{equation}
We write $\zeta(1+\eta+it) = R e^{i\phi}$. Thus $f_t(1+\eta) = \Re \Big(\zeta(1+\eta+it)^N\Big) = R^N \cos(N\phi)$. 
We choose a sequence of $N$'s such that $\displaystyle{\lim_{N\to\infty} N\phi = 0 \pmod {2\pi}}$.
Thus
\begin{multline}\label{bound1}
\log |f_t(1+\eta) | 
= N\log((1+o(1)) R) 
= N\log((1+o(1)) |\zeta(1+\eta+it)|) 
\\ \ge N \log\big( \frac{\zeta(2(1+\eta))}{\zeta(1+\eta)} \big)  +o_N(1),
\end{multline}
where $o_N(1)\rightarrow0$ when $N\rightarrow\infty$.
We now split the integral in the left term of inequality \eqref{bound0} depending on the sign of $\cos \theta$.
For $ \theta\in (-\pi/2,\pi/2) $, $\Re (1+\eta+(1+2\eta)e^{i\theta}\pm it )>1+\eta>1$, and we use the trivial bound
\[ 
\left| \zeta(1+\eta+(1+2\eta)e^{i\theta}\pm it ) \right|
\le  \zeta(1+\eta) ,
\]
giving
\begin{equation} \label{bound2}
\int_{-\pi/2}^{\pi/2} \log\left| f_t(1+\eta+(1+2\eta)e^{i\theta}) \right|d\theta \le N \pi \log \big( \zeta(1+\eta)\big).
\end{equation}
For $ \theta\in (\pi/2,3\pi/2) $, we use Rademacher's bound \cite[equation (7.4)]{Rad}:
\[
\left|\zeta(s)\right| \le 3\frac{|1+s|}{|1-s|} \Big(\frac{|1+s|}{2\pi}\Big)^{\frac{1+\eta-\Re s}2} \zeta(1+\eta)
\]
with $s=1+\eta+(1+2\eta)e^{i\theta}\pm it$.
Since
\[ |1+s|  \le  t+ 3(1+\eta),\ 
 |1-s|  \ge |\Im s| \ge t- (1+2\eta), \text{ and }
0\le  1+\eta-\Re s \le 1+2\eta,\]
then
\begin{multline}\label{bound3}
\int_{\pi/2}^{3\pi/2} \log\left| f_t(1+\eta+(1+2\eta)e^{i\theta}) \right|d\theta
\\ \le N \pi  \log\Big( 3\frac{t+3(1+\eta) }{t- (1+2\eta)} \Big(\frac{t+3(1+\eta)}{2\pi}\Big)^{\frac{1+2\eta}2} \zeta(1+\eta) \Big).
\end{multline}
Together with \eqref{bound0}, \eqref{bound1}, \eqref{bound2}, and \eqref{bound3}, we deduce
\begin{multline}
n 
\le \frac{N }{2\log2}  \log\Big( 3\frac{t+3(1+\eta) }{t- (1+2\eta)} \Big(\frac{3(1+\eta) +t}{2\pi}\Big)^{\frac{1+2\eta}2} 
\frac{\zeta(1+\eta)^4} {\zeta(2(1+\eta))^2} \Big)  +o_N(1) \\
\le \frac{N (1+2\eta) }{4\log2} \log t
+ \frac{N }{2\log2}  \log\Big( 3\frac{t+3(1+\eta) }{t- (1+2\eta)} \Big(\frac{3(1+\eta)/t +1}{2\pi}\Big)^{\frac{1+2\eta}2} 
\frac{\zeta(1+\eta)^4} {\zeta(2(1+\eta))^2} \Big)  +o_N(1).
\end{multline}
Together with \eqref{ineq-arg} and letting $N\to\infty$, we obtain
\begin{multline*}
 \left|  \int_{\s_0}^{\s_1} \arg\zeta(\tau+it)d\tau  \right|   
 \le 
\frac{\pi (1+2\eta) (\s_1-\s_0)}{4\log2}  \log t
\\+
\frac{\pi(\s_1-\s_0) }{2\log2}  \log\Big(   3\frac{t+3(1+\eta) }{t- (1+2\eta)} \Big(\frac{3(1+\eta)/t +1}{2\pi}\Big)^{\frac{1+2\eta}2} 
\frac{\zeta(1+\eta)^4} {\zeta(2(1+\eta))^2}  \Big).
\end{multline*}
Observing that the second term decreases with $t\ge H$ achieves the proof.
\end{proof}
%%%
\section{Explicit upper bounds for $N(\s,T)$ - Proof of Theorem \ref{main-density}}
\begin{proof}
We recall that $\s,\s_0,\s_1,H$ and $T$ satisfy \eqref{conditions}.
We consider the number $N(\s,T)$ of zeros $\varrho=\beta+i\gamma$ of zeta in the rectangle $\s< \beta <1$ and $H< \gamma <T$.
Since $N(\s,H)=0$, we have
\begin{equation}\label{boundN1}
N(\s,T)  
 \le \frac1{\s-\s_0} \int_{\s_0}^{\s_1} \Big( N(\tau,T) -N(\tau,H) \Big) d\tau .
\end{equation}
It follows from a lemma of Littlewood (see \cite[(9.9.1)]{Tit}) that
\[
 \int_{\s_0}^{\s_1} \Big( N(\tau,T) -N(\tau,H) \Big)  d\tau = - \frac1{2\pi i} \int_{\mathcal{R}} \log \zeta(s) ds ,
\]
where $\mathcal{R}$ is the rectangle with vertices $\s_0+iH$, $\s_1+iH$, $\s_1+iT$, and $\s_0+iT$.
Thus
\begin{multline}\label{boundN2}
N(\s,T)  
\le   \frac{1}{2\pi(\s-\s_0)}  \Big(
\int_{H}^{T} \log| \zeta(\s_0+it)| dt
+   \int_{\s_0}^{\s_1} \arg \zeta(\tau+iT) d\tau 
\Big. \\ \Big.
- \int_{\s_0}^{\s_1} \arg \zeta(\tau+iH) d\tau 
  - \int_{H}^{T} \log |\zeta(\s_1+it)| dt
\Big).
\end{multline}
We use Theorem \ref{moment-zeta}, Lemma \ref{bound-logzeta}, and Lemma \ref{bound-arg} respectively to bound these integrals:
\begin{align*}
& \int_{H}^{T} \log |\zeta(\s_0+it)| dt 
 \le  \frac{T-H}{2} \log \Big(  \zeta(2\s_0)  + \mathcal{E}_1(\s_0,H) \Big),
\\
& - \int_{H}^{T} \log |\zeta(\s_1+it)| dt \le \mathcal{E}_2 ,
\\
&
 \int_{\s_0}^{\s_1} \arg \zeta(\tau+iT) d\tau 
-  \int_{\s_0}^{\s_1} \arg \zeta(\tau+iH) d\tau
\le  \mathcal{E}_3(\s_0) \log  (HT) +  \mathcal{E}_4(\s_0,H) ,
\end{align*}
where the $\mathcal{E}_i$'s are defined respectively in \eqref{def-E1}, \eqref{def-E2}, \eqref{def-E3}, and  \eqref{def-E4}.
We obtain
\[N(\s,T)  
\le b_1(\s_0,H) (T-H) + b_2(\s_0,H)   \log (TH) + b_3(\s_0,H),
\]
with
\begin{equation}\label{def-bi}
b_1(\s_0,H) = \frac{\log \Big(  \zeta(2\s_0)  + \mathcal{E}_1(\s_0,H) \Big) }{4\pi(\s-\s_0)}  ,\ 
b_2(\s_0,H)= \frac{\mathcal{E}_3(\s_0)}{2\pi(\s-\s_0)} ,\ 
b_3(\s_0,H)=\frac{\mathcal{E}_2 +\mathcal{E}_4(\s_0,H) }{2\pi(\s-\s_0)} . 
\end{equation}
It follows
\[N(\s,T)  
\le c_1  T + c_2   \log T + c_3 ,
\ \text{with}\ 
c_1 = b_1  ,\ 
c_2 = b_2 ,\ 
c_3 = -b_1 H + b_2 \log H + b_3.
\]
\end{proof}
\clearpage
Table \ref{table1} records values of the $b_i$'s and $c_i$'s computed for $H_0=3.061\cdot 10^{10}$. 
Specific choices of parameters $\sigma_0$ and $H$ are chosen in order to obtain good bounds for $N(\s,T)$ when $T$ is asymptotically large. 
The values of $\s_0,b_1,c_1,b_2$, and $c_2$ displayed in the table are rounded up to $4$ decimal places. We take the ceiling of the values of $H,b_3$, and $c_3$.
\\
\begin{table}[h!]
\centering
\caption{$N(\s,T) \le b_1  (T-H) + b_2  \log (TH) + b_3 $ and $N(\s,T) \le c_1  T + c_2  \log T + c_3$, for $T\ge H_0$.}
\label{table1}
\begin{tabular}{|r|r|r|r|r|r|r|}
\hline
$\s$ & $\s_0$ & $H$ & $b_1=c_1 $ & $b_2=c_2 $ & $b_3 $& $c_3 $
\\
\hline
0.60 & 0.5229 & $19\,399$     & 4.2288 & 2.2841 & 333 &  $-81\,673$\\ \hline
0.65 & 0.5552 & $40\,105$     & 2.4361 & 1.7965 & 262 &  $-97\,414$\\ \hline
0.70 & 0.5873 & $91\,470$     & 1.4934 & 1.4609 & 213 &  $-136\,370$\\ \hline
0.75 & 0.6096 & $169\,119$    & 1.0031 & 1.1442 & 167 &  $-169\,449$\\ \hline
0.76 & 0.6136 & $188\,973$    & 0.9355 & 1.0921 & 160 &  $-176\,604$\\ \hline
0.77 & 0.6175 & $210\,645$    & 0.8750 & 1.0437 & 153 &  $-184\,134$\\ \hline
0.78 & 0.6213 & $234\,346$    & 0.8205 & 0.9986 & 146 &  $-192\,120$\\ \hline
0.79 & 0.6250 & $260\,321$    & 0.7714 & 0.9566 & 140 &  $-200\,644$\\ \hline
0.80 & 0.6287 & $288\,853$    & 0.7269 & 0.9176 & 134 &  $-209\,795$\\ \hline
0.81 & 0.6324 & $320\,270$    & 0.6864 & 0.8812 & 129 &  $-219\,667$\\ \hline
0.82 & 0.6361 & $354\,951$    & 0.6495 & 0.8473 & 124 &  $-230\,367$\\ \hline
0.83 & 0.6398 & $393\,341$    & 0.6156 & 0.8157 & 119 &  $-242\,009$\\ \hline
0.84 & 0.6435 & $435\,955$    & 0.5846 & 0.7862 & 115 &  $-254\,724$\\ \hline
0.85 & 0.6472 & $483\,393$    & 0.5561 & 0.7586 & 111 &  $-268\,658$\\ \hline
0.86 & 0.6510 & $536\,357$    & 0.5297 & 0.7327 & 107 &  $-283\,978$\\ \hline
0.87 & 0.6548 & $595\,670$    & 0.5053 & 0.7085 & 104 &  $-300\,872$\\ \hline
0.88 & 0.6587 & $662\,291$    & 0.4827 & 0.6857 & 101 &  $-319\,555$\\ \hline
0.89 & 0.6626 & $737\,343$    & 0.4617 & 0.6644 & 97  &  $-340\,272$\\ \hline
0.90 & 0.6667 & $822\,142$    & 0.4421 & 0.6443 & 95  &  $-363\,301$\\ \hline
0.91 & 0.6708 & $918\,225$    & 0.4238 & 0.6253 & 92  &  $-388\,959$\\ \hline
0.92 & 0.6750 & $1\,027\,390$ & 0.4066 & 0.6075 & 89  &  $-417\,606$\\ \hline
0.93 & 0.6793 & $1\,151\,729$ & 0.3905 & 0.5906 & 87  &  $-449\,647$\\ \hline
0.94 & 0.6838 & $1\,293\,683$ & 0.3754 & 0.5747 & 84  &  $-485\,543$\\ \hline
0.95 & 0.6883 & $1\,456\,079$ & 0.3612 & 0.5596 & 82  &  $-525\,807$\\ \hline
0.96 & 0.6930 & $1\,642\,194$ & 0.3478 & 0.5452 & 80  &  $-571\,018$\\ \hline
0.97 & 0.6977 & $1\,855\,803$ & 0.3352 & 0.5316 & 78  &  $-621\,815$\\ \hline
0.98 & 0.7026 & $2\,101\,249$ & 0.3232 & 0.5187 & 76  &  $-678\,911$\\ \hline
0.99 & 0.7077 & $2\,383\,498$ & 0.3118 & 0.5063 & 74  &  $-743\,087$\\ \hline
\end{tabular}
\end{table} 
\clearpage
%%%
\noindent
{\bf Aknowledgments.}
I would like to thank Olivier Ramar\'e for his comments on this article.
%

%%%%%
\end{document}